\documentclass{article}

\usepackage{arxiv}

\usepackage[utf8]{inputenc} 
\usepackage[T1]{fontenc}    
\usepackage{hyperref}       
\usepackage{url}            
\usepackage{booktabs}       
\usepackage{amsfonts}       
\usepackage{nicefrac}       
\usepackage{microtype}      
\usepackage{amsmath}		
\usepackage{amsthm}         
\usepackage{graphicx}
\usepackage{natbib}
\usepackage{doi}

\usepackage[nameinlink]{cleveref}
\Crefname{equation}{Eq.}{Eqs.}
\Crefname{figure}{Fig.}{Figs.}
\Crefname{efigure}{Extended Data Fig.}{Extended Data Figs.}
\Crefname{tabular}{Tab.}{Tabs.}
\Crefname{linenum}{line}{lines}
\crefrangelabelformat{equation}{#3#1#4-#5#2#6}
\creflabelformat{equation}{#2#1#3}

\title{Eigenvalue spectrum support of\\paired random matrices with pseudo-inverse}

\author{ \href{https://orcid.org/0000-0001-8161-0430}{\includegraphics[scale=0.06]{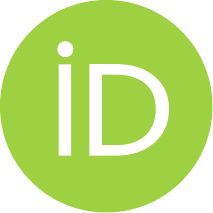}\hspace{1mm}Uri Cohen}\thanks{Corresponding author} \\
        Computational and Biological Learning Lab,\\
        Dept. of Engineering, University of Cambridge\\
        Cambridge, UK\\
        \texttt{uc231@cam.ac.uk}
}

\renewcommand{\shorttitle}{Paired matrices with pseudo-inverse}

\hypersetup{
pdftitle={Eigenvalue spectrum support of paired random matrices with pseudo-inverse},
pdfsubject={math.SP, math.PR},
pdfauthor={Uri Cohen},
pdfkeywords={Random Matrix Theory, Moore-Penrose pseudo-inverse, Eigenvalue Spectrum},
}

\newtheorem{definition}{Definition}
\newtheorem{theorem}{Theorem}
\newtheorem{corollary}{Corollary}

\newtheorem{conjecture}{Conjecture}

\begin{document}
\maketitle

\begin{abstract}
The Moore-Penrose pseudo-inverse $X^\dagger$, defined for rectangular matrices, naturally emerges in many areas of mathematics and science. For a pair of rectangular matrices $X, Y$ where the corresponding entries are jointly Gaussian and i.i.d., we analyse the support of the eigenvalue spectrum of $XY^\dagger$.
\end{abstract}


\section{Introduction}
We are interested in pairs of rectangular random matrices of equal size where their corresponding elements are independent, identically distributed (i.i.d.) from a 2-D joint distribution.
\begin{definition}[Real Paired Gaussian Matrices]\label{def1}
For $N, P\in\mathbb{N}$ and covariance matrix $\Sigma\in\mathbb{R}^{2\times2}$, Real Paired Gaussian Matrices are a pair of real rectangular matrices of $X, Y\in\mathbb{R}^{N\times P}$ where corresponding entries $x=X_{i\mu},y=Y_{i\mu}$ for any $i=1\ldots N$, $\mu=1\ldots P$, are jointly i.i.d.\ Gaussian $\left(x,y\right)\sim{\cal N}\left(0,\Sigma/N\right)$.
\end{definition}

\begin{definition}[Complex Paired Gaussian Matrices with independent components]\label{def2}
For $N, P\in\mathbb{N}$, $\Sigma$ as above, Complex Paired Gaussian Matrices with independent components are complex rectangular matrices $X, Y\in \mathbb{C}^{N\times P}$ such that $\mathrm{Re}\left(X\right), \mathrm{Re}\left(Y\right)$ are real paired Gaussian matrices with covariance $\Sigma_{Re}$, $\mathrm{Im}\left(X\right), \mathrm{Im}\left(Y\right)$ are real paired Gaussian matrices with covariance $\Sigma_{Im}$, and satisfy $\Sigma_{Re}+\Sigma_{Im}=\Sigma$ (i.e., the real and imaginary components are independent). 
\end{definition}

\begin{definition}[Complex Paired Gaussian Matrices]\label{def3}
For $N, P\in\mathbb{N}$, covariance matrix $\Gamma\in\mathbb{R}^{4\times4}$,
Complex Paired Gaussian Matrices are a pair of complex rectangular matrices of $X, Y\in\mathbb{C}^{N\times P}$ where corresponding entries $x=X_{i\mu},y=Y_{i\mu}$ 
are jointly i.i.d.\ Gaussian $\left(\mathrm{Re}x,\mathrm{Im},\mathrm{Re}y,\mathrm{Im}y\right)\sim{\cal N}\left(0,\Gamma/N\right)$.
\end{definition}

Note that \Cref{def2} generalises \Cref{def1} as it correspond to the case of $\Sigma_\mathrm{Im} = 0$. Without loss of generality,
$$\Sigma=\mathrm{Var}\left(x,y\right)=\begin{pmatrix}
\sigma_{x}^{2} & \tau\sigma_{x}\sigma_{y}\\
\bar{\tau}\sigma_{x}\sigma_{y} & \sigma_{y}^{2}
\end{pmatrix}$$
for $\sigma_x,\sigma_y\in\mathbb{R}^+$ and $\left|\tau\right|\le1$, where $\tau\in\mathbb{R}$ for the first two definitions and $\tau\in\mathbb{C}$ for the third definition.

Denoting the dimensions ratio $\alpha=P/N$, we consider a matrix $M\in\mathbb{C}^{N\times N}$ defined from paired Gaussian matrices $X, Y$ using either a conjugate transpose $M=XY^*$ (a scenario previously discussed under the name ``non-Hermitian Wishart ensemble'' \cite{akemann2021non}) or a pseudo-inverse \cite{penrose1955generalized} $M=XY^{\dagger}$ and wish to calculate the support of the limiting spectral density of $M$ in terms of $\sigma_x,\sigma_y,\tau,\alpha$, namely the set with positive density for $N\to\infty$. 

\section{Results}
We denote the empirical spectral density of a matrix $M_N\in\mathbb{C}^{N\times N}$ as $\mu^N_M\left(\omega\right)=\frac{1}{N}\sum_i^N\delta\left(\omega-\lambda_i\left(M_N\right)\right)$. If a series of such measures converges weakly to a limiting spectral density, we denote it $\mu_M\left(\omega\right)$, and denote its 
support ${\cal S}_M=\left\{\omega:\mu_M\left(\omega\right)>0\right\}$, with a slight abuse of notation, as both $\mu_M$ and ${\cal S}_M$ are not defined for a specific matrix $M$.

\begin{theorem}\label{th1}
For complex paired Gaussian matrices with independent components $X, Y$, the support of the limiting spectral density of $M=XY^*$ is:
\begin{equation}\label{eq:support1}
    {\cal S}_{XY^*} = \left\{0\right\}_{\alpha<1}\cup\left\{\lambda:\left(\frac{\mathrm{Re}\lambda-\sigma_x\sigma_y\left(1+\alpha\right)\mathrm{Re}\left(\tau\right)}{\sigma_x\sigma_y\sqrt{\alpha}\left(1+\left|\tau\right|^2\right)}\right)^2 + \left(\frac{\mathrm{Im}\lambda-\sigma_x\sigma_y\left(1+\alpha\right)\mathrm{Im}\left(\tau\right)}{\sigma_x\sigma_y\sqrt{\alpha}\left(1-\left|\tau\right|^2\right)}\right)^2\le1\right\}
\end{equation}
\end{theorem}
The support is an ellipsoid if $\alpha\ge1$, or a union thereof with $0$ if $\alpha<1$. For the case discussed here, $\tau\in\mathbb{R}$ so that $\mathrm{Im}\left(\tau\right)=0$, but this equation is valid in a more general case: for a complex $\tau$ the support is rotated by $\arg\left(\tau\right)$.

\begin{theorem}\label{th2}
For complex paired Gaussian matrices $X, Y$, the support of the limiting spectral density of $M=XY^*$ is $e^{i \arg\left(\tau\right)}{\cal S}_{XY^*}=\left\{\lambda:e^{-i \arg\left(\tau\right)}\lambda\in {\cal S}_{XY^*}\right\}$.
\end{theorem}

\begin{theorem}\label{th3}
For paired Gaussian matrices $X, Y$ with $\alpha\ne1$, denoting $\beta=\max\left\{1/\alpha,\alpha\right\}$, the support of the limiting spectral density of $M=XY^\dagger$ is:
\begin{equation}\label{eq:support2}
    {\cal S}_{XY^\dagger} = \left\{0\right\}_{\alpha<1}\cup\left\{\lambda:\left|\lambda-\frac{\sigma_x}{\sigma_y}\tau\right|^2
    \le\frac{\sigma_x^2}{\sigma_y^2}\frac{1-\left|\tau\right|^2}{\beta-1}\right\}
\end{equation}
\end{theorem}
The support is a circle if $\alpha>1$, or a union thereof with $0$ if $\alpha<1$, and the case $\alpha=1$ is not covered. 

\begin{conjecture}\label{conj:support}
For paired Gaussian matrices, the support of $\mu_{XY^*}^N$ converges to the support of the limiting $\mu_{XY^*}$.
\end{conjecture}
This corresponds to the lack of isolated outliers for $XY^*$. For example, this property has been proven for the Ginibre ensemble, where it was further demonstrated that outliers can be created using bounded rank perturbations \cite{tao2013outliers}.
This might be proven by showing that the Brown measure is continuous with respect to the topology of convergence \cite{belinschi2018eigenvalues}.

\section{Proofs}
\begin{proof}[Proof of \Cref{th1} for $\alpha\ge1$] This case is the main result of \cite{akemann2021non}, where it is derived (using a different notation) for $\alpha\ge1$  that $\mu^N_{XY^*}$ converges weakly to a limiting spectral density $\mu_{XY^*}$ with the specified support, with the additional assumptions that $\sigma_x=\sigma_y=1$, and $\Sigma_{Re}=\Sigma_{Im}$. Because the resulting eigenvalues scale multiplicatively with $\sigma_x \sigma_y$, \Cref{eq:support1} is obtained from Eq.\ 1.7 in \cite{akemann2021non} by scaling $\lambda$ into $\lambda/\sigma_x\sigma_y$. Furthermore, their result depends only on $\tau$, the off-diagonal term of $\Sigma=\Sigma_{Re}+\Sigma_{Im}$, and thus generalises to any choice of $\Sigma_{Re}, \Sigma_{Im}$, as in our definitions.
\end{proof}
\begin{proof}[Proof of \Cref{th1} for $\alpha<1$] We note the characteristic polynomial of $XY^*$ can be related to that of $Y^*X$ by the Weinstein-Aronszajn identity $p_{XY^*}\left(x\right)=\det\left(xI-XY^*\right)=x^{N-P}\det\left(xI-Y^*X\right)=x^{N-P}p_{Y^*X}\left(x\right)$ so the eigenvalues of $XY^*$ are $N-P$ zeros, and the $P$ eigenvalues of $Y^*X$. This relation holds exactly for a finite $N$, $\mu^N_{XY^*}\left(\lambda\right)=\left(1-\alpha\right)\delta\left(\lambda\right)+\alpha\mu^P_{XY^*}\left(\lambda\right)$, and thus also for the limiting spectral density. The measure $\mu_{XY^*}$ is supported, according to the first half of the proof, at the following ellipsoid from \Cref{eq:support1} with dimensions ratio $1/\alpha>1$, and additional scaling of $\sigma_x\sigma_y$ to $\frac{P}{N}\sigma_x\sigma_y$ due to correcting the scaling from $\Sigma/N$ in \Cref{def1} into $\Sigma/P$. Those terms cancel, and the ellipsoid support from \Cref{eq:support1} is the same in both cases.
\end{proof}

This also provides the exact limiting spectral density of $XY^*$ for $\alpha<1$, in terms of the known result for $\alpha>1$ \cite{akemann2021non}.

\begin{proof}[Proof of \Cref{th2}] We note it is possible to diagonalise $\Gamma$, a $4\times 4$ positive-definite matrix, using three rotation operations, one applied to components of $x$, one applied to the components of $y$, and one applied at the $2\times 2$ block structure. The latter is equivalent to multiplication by a complex scalar $c$. The former are equivalent to multiplying the complex $x$ (respectively $y$) by a constant $c_x$ (respectively $c_y$) such that $c_x x$ (respectively $c_y y$) are complex Gaussian variables with independent real and imaginary components. As all the components of $X, Y$ are identically distributed, $c c_x\bar{c_y} X Y^*$ satisfy \Cref{def2} and hence their support is given by \Cref{eq:support1}. Furthermore, this constant can be calculated as $\arg\left(\tau\right)$; the norms of the constants would not affect this normalised quantity. Finally, the effect of this multiplication is a rotation of the support around $0$, so that the centre of the ellipsoid moves from $\sigma_x\sigma_y\left(1-\alpha\right)\tau+i0$ in the independent case to $\sigma_x\sigma_y\left(1-\alpha\right)\left(\mathrm{Re}\left(\tau\right)+i\mathrm{Im}\left(\tau\right)\right)$ in the general case, as well as rotation of each $\lambda$ into $e^{i \arg\left(\tau\right)}\lambda$.
\end{proof}

The following corollary can be drawn from \Cref{eq:support1}, which we will use below. It was already noted in \cite{akemann2021non} for $\tau\in\mathbb{R}$.
\begin{corollary}\label{corollary1}For paired Gaussian matrices, 
    $0\in{\cal S}_{XY^*}$ iff $\left|\tau\right|^2\le1/\alpha$.
\end{corollary}
\begin{proof}
    The rotation in \Cref{th2} is around $0$, so the condition is the same for complex paired Gaussian matrices with or without independent components. For $\alpha<1$, the statement is trivial as both terms are true by definition. For $\alpha\ge1$, the left term becomes $\left|\tau\right|\left(1+\alpha\right)\le\sqrt{\alpha}\left(1+\left|\tau\right|^2\right)$ and denoting $g\left(x\right)=\frac{x}{1+x^2}$ yields $g\left(\left|\tau\right|\right)\le g\left(1/\sqrt{\alpha}\right)$. Thus, $\left|\tau\right|\le1/\sqrt{\alpha}$ from monotonicity of $g\left(x\right)$ for $x\in\left[0,1\right]$.
\end{proof}

We prove \Cref{th3} by showing how the condition $\lambda\in {\cal S}_{XY^\dagger}$ can be reduced to the condition $0\in Y^*Z_\lambda$, , for some matrix $Z_\lambda$, which we already understand from \Cref{corollary1}. The proof requires the yet unproven \Cref{conj:support}.

\begin{proof}[Proof of \Cref{th3} for $\alpha<1$, assuming \Cref{conj:support}] 
In this case $Y^{\dagger}=\left(Y^{*}Y\right)^{-1}Y^{*}$, and from the generalised matrix determinant lemma \cite{brookes2005matrix}, the characteristic polynomial of $XY^{\dagger}$ would be $p_{XY^\dagger}\left(x\right)=\det\left(x I-X\left(Y^{*}Y\right)^{-1}Y^{*}\right)	=\det\left(Y^{*}Y-\frac{1}{x}Y^{*}X\right)\frac{x^N}{\det\left(Y^{*}Y\right)}$ where $\det\left(Y^{*}Y\right)$ is a finite, strictly positive value for $\alpha<1$ from Marchenko-Pastur \cite{marchenko1967distribution}. We note that when $x\to0$ the determinant is dominated by $x^{-P}$ and the characteristic polynomial would have $x^{N-P}$. Thus, at least $N-P$ of the eigenvalues of $XY^\dagger$ are $0$, and this value is included in the support. For $0\ne\lambda\in \mathrm{E.V.}\left(XY^\dagger\right)$ we have that it satisfies $0=\det\left(Y^* Z\right)$ for $Z=Y-X/\lambda$.
Now note that for a fixed $\lambda\ne0$, we can consider a series of $P\times P$ matrices $M_P=Y^* Z$ where $Y, Z$ are paired Gaussian matrices, with dimensions ratio $1/\alpha$. Assuming \Cref{conj:support}, the support of the eigenvalue spectrum of $M_P$ converges for $P\to\infty$ to the support of the limiting density \Cref{eq:support1}, so except for a set whose measure vanishes, by \Cref{corollary1} it is strictly positive for $\left|\tau_\lambda\right|^2<\alpha$, and $0$ otherwise, where $\tau_\lambda=\mathrm{corrcoef}\left(y, y-x/\lambda\right)$.
Using the joint distribution of $x,y$:
\begin{equation}\label{eq:tau_vs_alpha}
    \left|\tau_{\lambda}\right|^{2}	
    =\frac{\left\langle \delta \bar{y}\delta\left(y-x / \lambda\right)\right\rangle \left\langle \delta y\delta\overline{y-x / \lambda}\right\rangle }{\left\langle \delta y\delta \bar{y}\right\rangle \left\langle \delta\left(y-x / \lambda\right)\delta\overline{y-x / \lambda}\right\rangle }	=\frac{\sigma_{y}^{2}-2\sigma_{x}\sigma_{y}\mathrm{Re}\left(\tau/\lambda\right)+\left|\tau\right|^{2}\sigma_{x}^{2}/\left|\lambda\right|^{2}}{\sigma_{y}^{2}-2\sigma_{x}\sigma_{y}\mathrm{Re}\left(\tau/\lambda\right)+\sigma_{x}^{2}/\left|\lambda\right|^{2}}
\end{equation}
and substituting $\mathrm{Re}\left(\tau/\lambda\right)=\left(\mathrm{Re}\tau\mathrm{Re}\lambda+\mathrm{Im}\tau\mathrm{Im}\lambda\right)/\left|\lambda\right|^2$ the condition on $\lambda$ becomes:
\begin{equation}\label{eq:resulting_support}
    \left(1-\alpha\right)\left|\lambda\right|^{2}\sigma_{y}^{2}-2\left(1-\alpha\right)\sigma_{x}\sigma_{y}\left(\mathrm{Re}\tau\mathrm{Re}\lambda+\mathrm{Im}\tau\mathrm{Im}\lambda\right) + \left(\left|\tau\right|^{2}-\alpha\right)\sigma_{x}^{2}\le 0
\end{equation}
which can be rewritten as a circular law $\left|\lambda-c\right|^2 \le r^2$ with a centre $c=\tau\frac{\sigma_{x}}{\sigma_{y}}$ and square radius $r^2=\frac{\sigma_{x}^{2}}{\sigma_{y}^{2}}\left(1-\tau^{2}\right)\frac{\alpha}{1-\alpha}$, so \Cref{eq:support2} follows for $\beta=1/\alpha$.
\end{proof}

\begin{proof}[Proof of \Cref{th3} for $\alpha>1$, assuming \Cref{conj:support}]
In this case $Y^{\dagger}=Y^{*}\left(YY^{*}\right)^{-1}$, and from the generalised matrix determinant lemma \cite{brookes2005matrix} the characteristic polynomial is $p_{XY^\dagger}\left(x\right)=\det\left(x I-XY^{*}\left(YY^{*}\right)^{-1}\right)	=\det\left(YY^{*}-\frac{1}{x}XY^{*}\right)\frac{x^N}{\det\left(YY^{*}\right)}$. 
It is not expected to have zeros at $x=0$, as the determinant would contribute $x^{-N}$ for $x\to0$. For $0\ne\lambda\in \mathrm{E.V.}\left(XY^\dagger\right)$, we have that it satisfies $0=\det\left(Z Y^*\right)$ for $Z=Y-X/\lambda$, and the argument continues as in $\alpha<1$. Here,
for a fixed $\lambda\ne0$, we can consider a series of $N\times N$ matrices $M_N=Z Y^*$ where $Y, Z$ are paired Gaussian matrices, with dimensions ratio $\alpha$ (instead of $1/\alpha$ in the $\alpha<1$ case). 
Assuming \Cref{conj:support}, the support of the eigenvalue spectrum of $M_N$ converges for $N\to\infty$ to the support of the limiting density \Cref{eq:support1}, so except for a set whose measure vanishes, by  \Cref{corollary1} it is strictly positive for $\left|\tau_\lambda\right|^2<\alpha$, and $0$ otherwise, so that \Cref{eq:tau_vs_alpha} is unmodified and \Cref{eq:resulting_support} has $1/\alpha$ terms instead of $\alpha$ terms. \Cref{eq:support2} follows with $\beta=\alpha$.
\end{proof}

We note that the above approach for \Cref{th3} does not apply to $\alpha=1$, as $Y^\dagger=Y^{-1}$ in this case. 

\section*{Acknowledgements}
I am grateful to Ofer Zeitouni for his valuable feedback on the manuscript and for clarifying the assumptions which need to be made in \Cref{conj:support}. I would like to thank Máté Lengyel and Yashar Ahmadian for many useful discussions and for their guidance and support. This work was supported by the Blavatnik Cambridge Postdoctoral Fellowships.
\clearpage
 


\bibliographystyle{unsrt}
\bibliography{references}
\end{document}